\documentclass{article}
\usepackage{authblk}
\usepackage{etex}
\usepackage[utf8]{inputenc}
\usepackage[english]{babel}
\usepackage{amsmath}
\usepackage{amsthm}
\usepackage{amssymb}
\usepackage{sectsty}
\usepackage{titlesec}
\usepackage{color}
\usepackage[color,matrix,arrow]{xy}
\usepackage{amsgen}
\usepackage{amstext}
\usepackage{amsbsy}
\usepackage{amsopn}
\usepackage{amsfonts}
\usepackage{eepic}
\usepackage{graphicx}
\usepackage{epsf}
\usepackage{pstricks}
\usepackage{float}
\usepackage{enumerate}
\usepackage{eqnarray}
\usepackage{faktor}
\xyoption{all}
\usepackage{hyperref}
\usepackage{thmtools}
\usepackage{thm-restate}

\usepackage{pgf}
\usepackage{tikz-cd}
\usetikzlibrary{automata, arrows.meta, positioning,decorations.pathmorphing}

\newcommand{\footrecall}[1]{%
} 
\usetikzlibrary{snakes,shapes,arrows,automata}

\titleformat*{\section}{\large\bfseries}
\titleformat*{\subsection}{\normalsize \bfseries}

	\newtheorem*{classes}{Theorem \ref{thm:classes}}
	\newtheorem*{density}{Theorem \ref{thm:density-one}}
\newtheorem{theorem}{Theorem}[section]
\newtheorem{proposition}[theorem]{Proposition}
\newtheorem{corollary}[theorem]{Corollary}
\newtheorem{lemma}[theorem]{Lemma}
\newtheorem{conjecture}[theorem]{Conjecture}
\newtheorem{question}[theorem]{Question}
\theoremstyle{definition}
\newtheorem{definition}[theorem]{Definition}

\theoremstyle{remark}
\newtheorem{remark}[theorem]{Remark}

\DeclareMathOperator{\ordT}{ord}

\newcommand{\N}{\mathbb N}
\newcommand{\Z}{\mathbb Z}
\newcommand{\Q}{\mathbb Q}

\newcommand{\floor}[1]{\left\lfloor #1\right\rfloor}
\newcommand{\fracpart}[1]{\left\{#1\right\}}
\newcommand{\T}{\mathcal T}
\newcommand{\Rset}{\mathcal R}

\title{Rational dynamics of a  prime-representing map}
\author{André Carvalho}
\affil{Research Center in Mathematics and Applications (CIMA)

	Department of Mathematics, School of Sciences and Technology of the University of \'Evora
	
	Rua Rom\~ao Ramalho, 59, 7000-671 \'Evora, Portugal
	
	\texttt{andrecruzcarvalho@gmail.com}\\}
\date{}

\begin{document}
	\maketitle
	
	\begin{abstract}
	 We study the rational dynamics of the map $\mathcal{T}(x)=\lfloor x\rfloor(1+\{x\})$,
	 which appears in the recursive construction of the prime-representing constant of
	 Fridman, Garbulsky, Glecer, Grime and Florentin. For a rational number $x\geq 2$
	 with denominator $M$, we define its order to be the least non-negative integer $n$ such
	 that $\mathcal{T}^n(x)$ is an integer, if such an $n$ exists, and ask whether every
	 rational number has finite order.
	 
	 For each \(n\), we prove that the reduced fractions \(a/M\) of exact order
	 \(n\) are described by residue classes of \(a\) modulo \(M^{n+1}\), and give a recurrence for the number $A(n,M)$ of
	 residue classes of exact order $n$. We then show that for each fixed
	 denominator the fractions of finite order have natural density one among all reduced
	 fractions with that denominator, which implies in particular that there is no infinite
	 arithmetic progression of rational numbers of infinite order. We also give an explicit
	 family of fractions of prescribed order for every denominator, and fully characterize
	 the case $M=2$.  
	\end{abstract}
	
	\section{Introduction}
	It was shown in \cite{FGGGF} that the recursive formula 
	\[
	f_{n+1}=\floor{f_n}\bigl(f_n-\floor{f_n}+1\bigr)
	\]
	can be used to generate all prime numbers in the following sense: there is an irrational number $x$ such that, if we put $f_1=x$, then the integer part of $f_n$ is the $n$-th prime.
	
	It is clear that if $f_n$ is an integer or  $f_n\leq 2$, then the sequence defined by the recursion stabilizes, that is $f_i=f_n$, for all $i\geq n$.
	Moreira pointed out in \cite{Moreira} that it is not hard to see that if the sequence  does not stabilize, then $f_1$ is irrational, but posed the question of whether it would always stabilize if $f_1$ is rational. That is supported by computational evidence.

		Let
	\[
	\T(x)=\floor{x}\bigl(1+\fracpart{x}\bigr),
	\]
	where $\fracpart{x}=x-\floor{x}$. So, if $f_n=x$, then $f_{n+1}=\T(x)$.

\begin{conjecture}\label{conj:global}
	Every rational number $x\in\Q_{\geq 2}$ has finite order under $\T$.
\end{conjecture}
	
	The present note concerns the rational dynamics of $\T$.  If $x\in\Q_{\geq 2}$,
	define
	\[
	\ordT(x)=\min\{n\geq 0: \T^n(x)\in\N\},
	\]
	with $\ordT(x)=\infty$ if no such $n$ exists.  The restriction $x\geq 2$ avoids the
	trivial fixed interval $[1,2)$, since if $1\leq x<2$ then $\T(x)=x$.

	We do not prove Conjecture~\ref{conj:global}.  Instead, we prove that for each
	fixed denominator $M$, the set of possible counterexamples has density zero.  This
	is directly analogous to the results of Azevedo, Carvalho and Machiavelo for the
	map $x\mapsto x\lceil x\rceil$ \cite{ACM}.

	To do so, we show that, for a fixed denominator $M$ and a rational number $a/M$ with $(a,M)=1$ and $a\geq 2M$, having order $n$, depends only on $a\bmod M^{n+1}$. Concretely, let \(M> 1\) and \(n> 0\).  
	We define $\Rset_{n,M}$ as the (eventually empty) set of integers $k$ such that $0\leq k< M^{n+1}$, $(k,M)=1$ and  for every integer \(a\geq 2M\) with \(
	a\equiv_{M^{n+1}} k,
	\)
	one has \(\ordT(a/M)=n\). We write
	\[
	A(n,M)=\#\Rset_{n,M}
	\qquad(M>1,\ n\geq 1).
	\]

We prove the following theorem:

	\begin{classes}
			Let \(M>1\) and \(n\geq 1\).  For every integer \(a\geq 2M\) with
		\((a,M)=1\),
		\[
		\ordT(a/M)=n
		\quad\Longleftrightarrow\quad
		a\bmod M^{n+1}\in \Rset_{n,M}.
		\]
		Moreover,
		\begin{equation*}
			A(n,M)
			=
			\varphi(M)\sum_{d\mid M} A(n-1,d)
			\left(\frac Md\right)^{n-1}.
		\end{equation*}
	\end{classes}

We then prove some consequences of this, such as fully characterizing the case $M=2$ and showing that all orders occur for every denominator by constructing an explicit
family of fractions of prescribed order for every denominator.

Let $A\subseteq \Z_{>0}$. The \emph{natural density} of $A$ is
\[
\delta(A)=\lim_{N\to\infty}\frac{\#\{a\in A: a\leq N\}}{N},
\]
if this limit exists. 
Let $B\subseteq \Z_{>0}$ have positive natural density, and let 
$A\subseteq B$. The \emph{natural density of $A$ relative to $B$} is
\[
\delta_B(A)=\lim_{N\to\infty}
\frac{\#\{a\in A: a\leq N\}}{\#\{b\in B: b\leq N\}},
\]
if this limit exists.  

We will show that, for a fixed denominator $M$, the set of numerators $a\geq 2M$ with $(a,M)=1$ for which $a/M$ has finite order has density $1$ relative to the set of possible numerators $b/M$ with $(b/M)=1$. This shows in some sense, that the probability of a rational number having finite order is $1$, providing some more evidence supporting Conjecture \ref{conj:global}.

\begin{density}
	For every $M\geq 1$,
	\[
	\sum_{n= 0}^{+\infty}\frac{A(n,M)}{\varphi(M^{n+1})}=1.
	\]
	Equivalently, among the integers $a\geq 2M$ with $(a,M)=1$, the set of $a$ for
	which $\ordT(a/M)<\infty$ has relative natural density \(1\) inside \(B_M=\{a\geq1:(a,M)=1\}\).
\end{density}

	\section{The image of a rational number}

	Let $M\geq 1$ be a fixed integer. We will consider the orbits of rational numbers of denominator $M$. We start by showing that denominators never increase and compute the result of a rational number after one iteration. This will be useful when we prove our results using induction.

	\begin{lemma}\label{lem:numerator}
		Let $M, a\geq 1$ be  integers and write $a=qM+r$ for some $0\leq r<M$. Then $\T(\frac{a}{M})= \frac{q(M+r)}{M}$		
	\end{lemma}
	\begin{proof}
		For $a\in \Z$, write
		\[
		a=qM+r,\qquad 0\leq r<M.
		\]
		Then
		\[
		\frac aM=q+\frac rM,
		\]
		and therefore
		\begin{equation}\label{eq:numerator-map}
			\T\left(\frac aM\right)
			=q\left(1+\frac rM\right)
			=\frac{q(M+r)}{M}.
		\end{equation}
	\end{proof}
	
	\begin{lemma} \label{lem:denominator-descent}
		Let $M\geq 1$ and let $a\in \Z$ be such that $(a,M)=1$.  Write $a=qM+r$ with $0\leq r<M$.  If $r\neq 0$,
		then the denominator of $\T(a/M)$ (in reduced form) is
		\[
		d=\frac{M}{(q,M)}.
		\]
		In particular, denominators never increase along rational orbits.
	\end{lemma}
	
	\begin{proof}
		Since $(a,M)=1$, the remainder $r$ and $M$ are coprime.  Hence, 
		\[
		(M+r,M)=(r,M)=1.
		\]
		By \eqref{eq:numerator-map},
		\[
		\T(a/M)=\frac{q(M+r)}{M}.
		\]
		Since $(M,M+r)=1$, the only common divisor between the numerator and $M$ comes from $q$.  Thus cancellation removes exactly $(q,M)$, yielding that  
		$$\frac{\frac{q(M+r)}{(q,M)}}{\frac{M}{(q,M)}}$$ is reduced.
	\end{proof}
	\begin{remark}
		If Conjecture~\ref{conj:global} is false, there is a rational orbit which never reaches an integer. This orbit must be strictly increasing, while
		its denominators form a non-increasing sequence of positive integers.  Hence any
		counterexample to Conjecture~\ref{conj:global} eventually has constant
		denominator.  
	More precisely, since the denominators form a non-increasing sequence 
	of positive integers, there exist $N\geq 0$ and $M>1$ such that 
	$\T^n(x)$ has denominator exactly $M$ for every $n\geq N$. Writing 
	$\T^n(x)=p_n/M$ in reduced form, we have $(p_n,M)=1$, and setting 
	$a_n=\lfloor p_n/M\rfloor$ and $r_n=p_n\bmod M$, we get
	\[
	\T^n(x)=a_n+\frac{r_n}{M},
	\qquad
	1\leq r_n<M,
	\qquad
	(r_n,M)=1,
	\]
	where the last condition follows from $(p_n,M)=1$. By 
	Lemma~\ref{lem:numerator},
	\[
	\T\left(a_n+\frac{r_n}{M}\right)=\frac{a_n(M+r_n)}{M}.
	\]
	The denominator of $\T^{n+1}(x)$ in 
	reduced form is $M/(a_n,M)$. For this to equal $M$, we need 
	$(a_n,M)=1$. Hence, for all $n\geq N$,
	\[
	(a_n,M)=(r_n,M)=1.
	\]
 
	\end{remark}
	
	\section{Describing rational numbers of a given order for a fixed denominator}
	
	For the rest of this section we fix a denominator \(M\).  We want to
	understand, for each \(n\), which reduced fractions with denominator \(M\)
	have order exactly \(n\).  Since the order will turn out to depend only on a
	congruence class of the numerator modulo $M^{n+1}$, it is convenient to introduce notation for
	those classes.
	
	Let \(M> 1\) and \(n> 0\).  
	We define $\Rset_{n,M}$ as the (possibly empty) set of integers $k$ such that $0\leq k< M^{n+1}$, $(k,M)=1$ and  for every integer \(a\geq 2M\) with \(
	a\equiv_{M^{n+1}} k,
	\)
	one has \(\ordT(a/M)=n\). Notice that, since $(k,M)=1$ and $a\equiv_{M^{n+1}} k$, then $(a,M)=1$. Hence, $a/M$ is in reduced form, and the condition $a\geq 2M$ ensures that $a/M\geq 2$, thus avoiding the interval $[1,2]$, where every number is fixed.
	
	The theorem below
	will show that \(\Rset_{n,M}\) encodes  all fractions of exact order \(n\), that is,
	for \(a\geq 2M\) and \((a,M)=1\), the value of \(\ordT(a/M)\), when equal to
	\(n\), is completely determined by the residue of \(a\) modulo \(M^{n+1}\).

	We shall write
	\[
	A(n,M)=\#\Rset_{n,M}
	\qquad(M>1,\ n\geq 1).
	\]
	For the recurrence it is useful to adopt the following  conventions:
	\[
	A(0,1)=1,
	\qquad
	A(0,M)=0\quad(M>1),
	\qquad
	A(n,1)=0\quad(n\geq 1).
	\]
	These conventions simply encode the degenerate cases.  A fraction with
	denominator \(1\) is already an integer, and therefore has order \(0\).  On the
	other hand, if \(M>1\) and \((a,M)=1\), then \(a/M\) is not an integer, so it
	cannot have order \(0\).
	
	\begin{theorem}\label{thm:classes}
		Let \(M>1\) and \(n\geq 1\).  For every integer \(a\geq 2M\) with
		\((a,M)=1\),
		\[
		\ordT(a/M)=n
		\quad\Longleftrightarrow\quad
		a\bmod M^{n+1}\in \Rset_{n,M}.
		\]
		Moreover,
		\begin{equation}\label{eq:recurrence}
				A(n,M)
			=
			\varphi(M)\sum_{d\mid M} A(n-1,d)
			\left(\frac Md\right)^{n-1}.
		\end{equation}

	\end{theorem}
	
	\begin{proof} We prove the statement by induction on $n$. We start by proving the case $n=1$.  We claim that
		\[
		\Rset_{1,M}
		=
		\{\,k:1\leq k<M,\ (k,M)=1\,\},
		\]
		where these integers are viewed as representatives modulo \(M^2\).
		
		Let $1\leq k < M$ with $(k,M)=1$. Let $a\geq 2M$ with $a\equiv_{M^2} k$.
		Then $a=qM^{2}+k$ and $a/M=qM+k/M$, so  $$\T(a/M)=qM\left(1+\frac kM\right)=qM+k,$$
		which is an integer, so $k\in \Rset_{1,M}.$ Conversely, let $k\in \Rset_{1,M}.$  We know, by definition, that $(k,M)=1$. We only have to see that $1\leq k<M$. Write $k=QM+c$ for $0\leq c<M$. Since $(k,M)=1$, then $c\neq 0$ and $(c,M)=1$. Suppose that $Q\neq 0$. Since $k<M^2$, then $Q<M$ and so, in particular, $M\nmid Q$.
		
		Choose $a>2M$ with $a\equiv_{M^2} k$. In particular $a\equiv_M c$, so $a=qM+c$ for some $q\equiv_M Q$. Indeed $a=k+tM^2=QM+c+tM^2$, so $qM=QM+tM^2$, so $q=Q+tM$. Hence $M\nmid q$ and $\T(a/M)=\frac{q(M+c)}{M}$ is not an integer because $(M+c,M)=1$ and $M\nmid q$. Therefore $\ordT(a/M)\neq 1$, which is absurd as $k\in \Rset_{1,M}$. We deduce that $Q=0$, so $k=c$ and $1\leq k<M$.

		We have proved that
		\[
		\Rset_{1,M}=
		\{\,k:1\leq k<M,\ (k,M)=1\,\}.
		\]
		Consequently,
		\[
		A(1,M)=\#\Rset_{1,M}=\varphi(M).
		\]
		
		Moreover, for every \(a\geq 2M\) with \((a,M)=1\),
		\[
		\ordT(a/M)=1
		\quad\Longleftrightarrow\quad
		a\bmod M^2\in\Rset_{1,M}.
		\]
		Indeed, writing \(a=qM+c\), the condition \(\ordT(a/M)=1\) is equivalent to
		\(M\mid q\), which is equivalent to \(a\equiv c\pmod{M^2}\), with
		\(1\leq c<M\) and \((c,M)=1\), i.e. \(a\bmod M^2\in \Rset_{1,M}\).

	Now let \(n\geq 2\), and assume that the result has been proved for all
	smaller positive orders.  We first prove the
	equivalence
	\[
	\ordT(a/M)=n
	\quad\Longleftrightarrow\quad
	a\bmod M^{n+1}\in\Rset_{n,M}.
	\]
	
	Suppose first that
	\[
	\ordT(a/M)=n.
	\]
	Let \(b\geq 2M\) be any integer such that
	\[
	b\equiv a\pmod{M^{n+1}}.
	\]
	We will prove that
	\[
	\ordT(b/M)=n.
	\]
	This will show that \(a\bmod M^{n+1}\in\Rset_{n,M}\).
	
	Write
	\[
	a=qM+c,
	\qquad
	b=q'M +c',
	\qquad
	0\leq c,c'<M.
	\]
	Since
	\(
	b\equiv_{M^{n+1}} a,
	\)
	we have in particular \(b\equiv a\pmod M\).  Hence
	\(
	c'=c.
	\)
	Also, since \((a,M)=1\), we have
	\[
	1\leq c<M,
	\qquad
	(c,M)=1.
	\]
	Moreover,
	\(
	b-a=(q'-q)M
	\)
	is divisible by \(M^{n+1}\).  Therefore
	\(
	q'\equiv_{M^n} q.
	\)
	
	Let
	\(
	h=(q,M).
	\)
	Since \(q'\equiv q\pmod M\), we also have
	\(
	(q',M)=h.
	\)
	Put
	\(
	N=\frac Mh.
	\)
	By Lemma~\ref{lem:denominator-descent}, 
	\[
	\T(a/M)=\frac{(q/h)(M+c)}{N}
	\qquad \text{ and } \qquad
	\T(b/M)=\frac{(q'/h)(M+c)}{N}
	\]
	are reduced.
	Since \(\ordT(a/M)=n\geq 2\), the first iterate of \(a/M\) is not an integer, that is, $	N>1$.	
	Now
	\(
	q'\equiv_{M^n} q.
	\)
	Dividing by \(h\), we get
	\[
	\frac{q'}h\equiv \frac qh\pmod{\frac{M^n}{h}}.
	\]
	Since
	\[
	\frac{M^n}{h}=NM^{n-1}
	\]
	and \(N\mid M\), we have
	\(
	N^n\mid NM^{n-1}.
	\)
	Hence
	\[
	\frac{q'}h\equiv \frac qh\pmod{N^n}.
	\]
	Multiplying by \(M+c\), we get
	\[
	\frac{q'}h(M+c)\equiv \frac qh(M+c)\pmod{N^n}.
	\]
	
	Since
	\(
	\ordT(a/M)=n,
	\)
	we have
	\(
	\ordT(\T(a/M))=n-1,
	\)
	that is,
	\[
	\ordT\left(\frac{(q/h)(M+c)}{N}\right)=n-1.
	\]
	
	We have that $(q/h)(M+c)$ and $N$ are coprime because $N\mid M$ and $(c,M)=1$,  and that $(q/h)(M+c)>2N$ (because $\T(a/M)>a/M\geq 2$), so, by the induction hypothesis, $(q/h)(M+c)\bmod N^n\in \Rset_{n-1,N}$. Since ${(q'/h)(M+c)}\equiv_{N^n} (q/h)(M+c)$, then, by definition of $\Rset_{n-1,N}$, 
	we have that 
	$\ordT(\T(b/M))=\ordT\left(\frac{(q'/h)(M+c)}{N}\right)=n-1$, that is, $\ordT(b/M)=n.$
		Thus every \(b\geq 2M\) with \(b\equiv a\pmod{M^{n+1}}\) has order \(n\).  Hence
	\[
	a\bmod M^{n+1}\in\Rset_{n,M}.
	\]

	Conversely, suppose that
	\(
	a\bmod M^{n+1}\in\Rset_{n,M}.
	\)
	By the definition of \(\Rset_{n,M}\), every integer \(b\geq 2M\) satisfying
	\(	b\equiv_{M^{n+1}} a
	\)
	has order \(n\).  In particular, taking \(b=a\), we get
	\(
	\ordT(a/M)=n.
	\)
	This proves the equivalence.
 
		It remains to prove the recurrence for $A(n,M)$.
		
		Each $k\in\Rset_{n,M}$ satisfies $0\leq k<M^{n+1}$ and $(k,M)=1$. Write
		\[
		k = qM+c, \qquad 0\leq c<M,\quad 0\leq q<M^n.
		\]
		Since $(k,M)=1$ we have $(c,M)=1$, giving $\varphi(M)$ admissible values 
		of $c$. We count, for each fixed $c$, the values of $q\in\{0,\ldots,M^n-1\}$ 
		for which $k\in\Rset_{n,M}$.
		
		By definitiomn, $k\in\Rset_{n,M}$ if and only if 
		$\ordT(a/M)=n$ for every $a\geq 2M$ with $a\equiv_{M^{n+1}} k$. Fix 
		any such $a$, and write $a=q_aM+c$ with $q_a\equiv_{M^n} q$. Let 
		$h=(q,M)=(q_a,M)$. Let $N=M/h$. 
		By Lemma~\ref{lem:denominator-descent},
		\[
		\T(a/M)=\frac{(q_a/h)(M+c)}{N}
		\]
		is reduced. Since $h=(q_a,M)$ and $N=M/h$, we have $(q_a/h,N)=1$. Since 
		$N\mid M$ and $(M+c,M)=(c,M)=1$, we have $(M+c,N)=1$. Hence 
		$(q_a/h)(M+c)$ and $N$ are coprime. Since $a/M$ is not an integer and 
		$a/M\geq 2$, we have $\T(a/M)>a/M\geq 2$, so $(q_a/h)(M+c)\geq 2N$.
		
		Since $q_a\equiv_{M^n} q$, dividing by $h$ gives $q_a/h\equiv_{M^n/h} q/h$. 
		Using $M=Nh$ we get $M^n/h=N^nh^{n-1}$, and since $N^n\mid N^nh^{n-1}$,
		\[
		\frac{q_a}{h}\equiv_{N^n}\frac{q}{h},
		\qquad\text{hence}\qquad
		\frac{q_a}{h}(M+c)\equiv_{N^n}\frac{q}{h}(M+c).
		\]
		Since $n\geq 2$, the condition $\ordT(a/M)=n$ is equivalent to 
		$\ordT(\T(a/M))=n-1$. But we have proved that
		\[
		\ordT\!\left(\frac{(q_a/h)(M+c)}{N}\right)=n-1
		\;\Longleftrightarrow\;
		\frac{q_a}{h}(M+c)\bmod N^n\in\Rset_{n-1,N}.
		\]
		Since $(q_a/h)(M+c)\equiv_{N^n}(q/h)(M+c)$, this condition is equivalent to
		\[
		\frac{q}{h}(M+c)\bmod N^n\in\Rset_{n-1,N},
		\]
		which depends only on $c$ and $q$, not on the choice of $a$. Since the 
		induction hypothesis gives an equivalence, every $q$ satisfying this 
		condition yields a $k=qM+c\in\Rset_{n,M}$, and every element of 
		$\Rset_{n,M}$ arises this way.
		
		We now count, for each $h\mid M$, the number of $q\in\{0,\ldots,M^n-1\}$ 
		with $(q,M)=h$ satisfying the condition. Write $q=hq'$ with $(q',N)=1$ 
		and $0\leq q'<M^n/h=N^nh^{n-1}$. The condition becomes 
		$q'(M+c)\bmod N^n\in\Rset_{n-1,N}$. Since $(M+c,N)=1$, multiplication 
		by $(M+c)$ is a bijection on $\Z_{N^n}^\times$, so the number of 
		residues $q'\bmod N^n$ with $(q',N)=1$ satisfying this condition is 
		exactly $A(n-1,N)$. Since $q'$ ranges over $\{0,\ldots,N^nh^{n-1}-1\}$, 
		it hits each residue class modulo $N^n$ exactly $h^{n-1}$ times, so the 
		total number of valid $q$ with $(q,M)=h$ is
		\[
		h^{n-1}\cdot A(n-1,N)=h^{n-1}\cdot A\!\left(n-1,\frac{M}{h}\right).
		\]
		Summing over all $h\mid M$ and substituting $d=M/h$ yields
		\[
		A(n,M)=\varphi(M)\sum_{h\mid M}h^{n-1}A\!\left(n-1,\frac{M}{h}\right)
		=\varphi(M)\sum_{d\mid M}A(n-1,d)\left(\frac{M}{d}\right)^{n-1},
		\]
		which completes the proof.
		\end{proof}

	As a corollary, we obtain a closed formula for the case where the denominator is a prime power.
	\begin{corollary} \label{a prime}
		Let $p$ be prime.  For every $n\geq 1$, and $s\geq 1$,
		\[
		A(n,p^s)=\binom{n+s-2}{s-1}\varphi(p^s)^n.
		\]
		In particular,
		\[
		A(n,p)=(p-1)^n.
		\]
	\end{corollary}
	
	\begin{proof}
		%For $s=1$, \eqref{eq:recurrence} yields that
		%\[
		%A(n,p)=\varphi(p)A(n-1,p)
		%\]
		%for $n\geq 2$, while $A(1,p)=\varphi(p)$.  Thus $A(n,p)=(p-1)^n$. Also, 
		We will prove the result by induction on $n$. 
		For $n=1$, we have from \eqref{eq:recurrence} that $A(1,p^s)=\varphi(p^s)=\binom{s-1}{s-1}\varphi(p^s),$ for all $s\geq 1$.
		
	 Let $n\geq 2$ and assume that, for smaller $n$, the result holds for any $s\geq 1$.  
		
		Since the divisors of  $p^s$ are $1,p,\dots,p^s$, we have, by \eqref{eq:recurrence}, that 
				\[
		A(n,p^s)=\varphi(p^s)
		\sum_{j=1}^s A(n-1,p^j)p^{(s-j)(n-1)}.
		\]
		By the induction hypothesis, we have that $A(n-1,p^j)=\binom{n+j-3}{j-1}\varphi(p^j)^{n-1}.$

		Since
		\[
		\varphi(p^s)\varphi(p^j)^{n-1}p^{(s-j)(n-1)}
		=\varphi(p^s)^n,
		\]
		the claim reduces to the hockey-stick identity
		\[
		\sum_{j=1}^s \binom{n+j-3}{j-1}
		=\binom{n+s-2}{s-1}.
		\]
	\end{proof}

	We record two explicit consequences of Theorem~\ref{thm:classes}.
	Given $m\in\N$, let $v_2(m)$ denote the exponent of $2$ in the prime factorization of $m$.
	
	\begin{proposition} 
		Let $a\geq 4$ be an odd integer.  Then, for every $n\geq 1$,
		\[
		\ordT(a/2)=n
		\quad\Longleftrightarrow\quad
		a\equiv 3+2^n\pmod{2^{n+1}}.
		\]
		In particular, $\ordT(a/2)=v_2(a-3)$.
	\end{proposition}
	
	\begin{proof}
		By Corollary \ref{a prime}, $A(n,2)=(2-1)^n=1$,
		so $\Rset_{n,2}$ consists of exactly one residue class modulo
		$2^{n+1}$.  By Theorem~\ref{thm:classes}, it suffices to identify
		this class.  Taking $a=3+2^n$, we have $a\geq 4$, $(a,2)=1$, and
		by Lemma~\ref{lem:numerator} one verifies directly by induction that
		$\T^n(a/2)$ is an integer while $\T^j(a/2)$ is not for $j<n$.
		Hence the unique class in $\Rset_{n,2}$ is $3+2^n\bmod 2^{n+1}$,
		giving the stated equivalence.  The formula $\ordT(a/2)=v_2(a-3)$
		follows since $a\equiv 3+2^n\pmod{2^{n+1}}$ is equivalent to
		$v_2(a-3)=n$.
	\end{proof}
	
	\begin{proposition}\label{prop:denom-two}
		Let $M\geq 2$ and $n\geq 1$.  Then $a=M+1+M^n(M-1)$ satisfies
		$(a,M)=1$ and $\ordT(a/M)=n$.  Consequently, for every fixed
		denominator $M\geq 2$, every positive integer occurs as the order
		of some reduced fraction with denominator $M$.
	\end{proposition}
	
	\begin{proof}
	Since $a\equiv 1\pmod{M}$, we have $(a,M)=1$.  Set $d=M+1$ and
	$a_0=a=d+M^n(M-1)$.  We claim that for $0\leq j\leq n$,
	\[
	\T^j(a_0/M)=\frac{a_j}{M},
	\qquad
	a_j=d+d^jM^{n-j}(M-1).
	\]
	The case $j=0$ is clear.  For $0\leq j<n$, since $d\equiv 1\pmod{M}$
	we have $a_j\equiv 1\pmod{M}$, so writing $a_j=qM+1$ with
	$q=1+d^jM^{n-j-1}(M-1)$, Lemma~\ref{lem:numerator} gives
	\[
	\T(a_j/M)=\frac{q(M+1)}{M}.
	\]
	Since $q(M+1)=d(1+d^jM^{n-j-1}(M-1))=d+d^{j+1}M^{n-j-1}(M-1)$, we verify the formula at step $j+1$.  For $j<n$, $a_j\equiv
	1\pmod{M}$ so $a_j/M$ is not an integer.  At $j=n$, since
	$d\equiv 1\pmod{M}$ we have $d^n\equiv 1\pmod{M}$, so
	\[
	a_n=d+d^n(M-1)\equiv 1+(M-1)=M\equiv 0\pmod{M},
	\]
	hence $a_n/M$ is an integer and $\ordT(a_0/M)=n$.
	\end{proof}
	
	 For a given denominator $M$ and positive integer $n$, we know that there are some numerators $a$ such that $\ordT(a/M)=n$. We pose the question of whether the smallest such $a$ can be described.
	\begin{question}
		Let $M\geq 2$ and define
		\[
		\mu(M,n)=\min\{a\geq 2M:(a,M)=1,\ \ordT(a/M)=n\}.
		\]
		Can we describe $\mu(M,n)$?
	\end{question}
	\section{Density of finite order rational numbers}
	The goal of this section is to prove that, for a given denominator $M$, the set of numerators $a$ such that $(a,M)=1$ and $a\geq 2M$ has density $1$ among all natural numbers that are coprime to $M$. We start by introducing the basic definitions.

	\begin{definition}
		Let $A\subseteq \Z_{>0}$. The \emph{natural density} of $A$ is
		\[
		\delta(A)=\lim_{N\to\infty}\frac{\#\{a\in A: a\leq N\}}{N},
		\]
		if this limit exists. 
	\end{definition}
	\begin{definition}
		Let $B\subseteq \Z_{>0}$ have positive natural density, and let 
		$A\subseteq B$. The \emph{natural density of $A$ relative to $B$} is
		\[
		\delta_B(A)=\lim_{N\to\infty}
		\frac{\#\{a\in A: a\leq N\}}{\#\{b\in B: b\leq N\}},
		\]
		if this limit exists.  
	\end{definition}
	
In our setting, the ambient set is $B_M=\{a\in\Z_{>0}: (a,M)=1\}$. 
%Since the residue classes coprime to $M$ are equidistributed, 
%$\delta(B_M)=\varphi(M)/M$, and among the first $N$ integers, 
%$\#\{a\leq N: (a,M)=1\}\sim \varphi(M)N/M$ as $N\to\infty$. We measure 
%density relative to $B_M$.

\begin{corollary}\label{cor:density}
	Let $M>1$, $n\geq 1$, and let
	\[
	S_n=\{a\geq 2M: (a,M)=1,\ \ordT(a/M)=n\}.
	\]
	Then $S_n$ has natural density relative to $B_M$ equal to
	\[
	\delta_{B_M}(S_n)=\frac{A(n,M)}{\varphi(M^{n+1})}.
	\]
\end{corollary}

\begin{proof}
	By Theorem~\ref{thm:classes}, $a\in S_n$ if and only if 
	$a\bmod M^{n+1}\in\Rset_{n,M}$. We compute the limit directly. For 
	any $N>0$,
	\[
	\#\{a\in S_n: a\leq N\}
	=\#\{a\leq N: (a,M)=1,\ a\bmod M^{n+1}\in\Rset_{n,M}\}.
	\]
	Now, $\Rset_{n,M}$ consists of $A(n,M)$ residue classes modulo 
	$M^{n+1}$, so between $iM^{n+1}+1$ and $(i+1)M^{n+1}$ there are exactly $A(n,M)$ integers belonging to $S_n$. Hence,
	\[
	 A(n,M)\cdot\left(\frac{N}{M^{n+1}}-1\right)\leq \#\{a\in S_n: a\leq N\}\leq A(n,M)\cdot\left(1+\frac{N}{M^{n+1}}\right).
	\]
	Similarly, between $iM+1$ and $(i+1)M$, there are exactly $\varphi(M)$ integers coprime to $M$, so 
	\[
	\varphi(M)\left(\frac{N}M-1\right)\leq \#\{a\in B_M: a\leq N\}\leq \varphi(M)\left(\frac{N}M+1\right).
	\]
	
	Therefore, 
	$$\frac{A(n,M)\cdot\left(\frac{N}{M^{n+1}}-1\right)}{\varphi(M)\left(\frac{N}M+1\right)}\leq \frac{\#\{a\in S_n:a\leq N\}}{\#\{a\in B_M:a\leq N\}}\leq \frac{A(n,M)\cdot\left(1+\frac{N}{M^{n+1}}\right)}{	\varphi(M)\left(\frac{N}M-1\right)}$$
	
As $N\to\infty$, both bounds tend to 
$\dfrac{A(n,M)/M^{n+1}}{\varphi(M)/M}=\dfrac{A(n,M)}{M^n\varphi(M)}
=\dfrac{A(n,M)}{\varphi(M^{n+1})}$,
so by the squeeze theorem the limit exists and equals the claimed value.
 
\end{proof}
 
	\begin{theorem}\label{thm:density-one}
		For every $M\geq 1$,
		\[
		\sum_{n= 0}^{+\infty}\frac{A(n,M)}{\varphi(M^{n+1})}=1.
		\]
		Equivalently, among the integers $a\geq 2M$ with $(a,M)=1$, the set of $a$ for
		which $\ordT(a/M)<\infty$ has relative natural density \(1\) inside \(B_M=\{a\geq1:(a,M)=1\}\).
	\end{theorem}
	
	\begin{proof}
		Let
		\[
		P(M)=\sum_{n= 0}^{+\infty} \frac{A(n,M)}{\varphi(M^{n+1})}.
		\]
	Since each integer $a$ has at most one order, the sets $S_n=\{a\geq 2M: 
	(a,M)=1,\,\ordT(a/M)=n\}$ are pairwise disjoint. By 
	Corollary~\ref{cor:density}, the natural density of $S_n$ within the 
	integers coprime to $M$ is $A(n,M)/\varphi(M^{n+1})$, and since 
	densities of disjoint sets sum to at most $1$, we have $P(M)\leq 1$.  We prove $P(M)=1$ by induction on $M$.
		
		For $M=1$, we have $A(0,1)=1$ and $A(n,1)=0$ for all $n\geq 1$, so $P(1)=1$.
		Assume $M>1$ and $P(d)=1$ for every proper divisor $d$ of $M$.  Since
		\[
		\varphi(M^{n+1})=M^n\varphi(M),
		\]
		Theorem~\ref{thm:classes} gives
		\begin{align*}
			P(M)
			&=\sum_{n= 1}^{+\infty}
			\frac{\varphi(M)}{M^n\varphi(M)}
			\sum_{d\mid M}A(n-1,d)\left(\frac Md\right)^{n-1} \\
			&=\frac1M\sum_{d\mid M}\sum_{n= 0}^{+\infty}\frac{A(n,d)}{d^n} \\
			&=\frac1M\sum_{d\mid M}\varphi(d)P(d).
		\end{align*}
		Separating the term $d=M$ yields
		\[
		P(M)=\frac{\varphi(M)}M P(M)
		+\frac1M\sum_{\substack{d\mid M\\ d<M}}\varphi(d)P(d).
		\]
		By the induction hypothesis, this becomes
		\[
		P(M)=\frac{\varphi(M)}M P(M)
		+\frac1M\sum_{\substack{d\mid M\\ d<M}}\varphi(d).
		\]
		Using $\sum_{d\mid M}\varphi(d)=M$, we get
		\[
		\left(1-\frac{\varphi(M)}M\right)P(M)
		=1-\frac{\varphi(M)}M.
		\]
		Since $M>1$, this implies $P(M)=1$.
	\end{proof}

As noted in \cite{ACM}, the previous theorem yields the following corllary immediately:
\begin{corollary}
There is no infinite arithmetic progression in $\Q$ whose elements
have infinite order.
\end{corollary}

 \section*{Acknowledgements}
 The author is grateful to Bruno Giordano for having introduced Conjecture \ref{conj:global} to the author, and to António Machiavelo for some discussions about \cite{ACM}.
  The  author was supported by national funds through the Funda\c c\~ao
  para a Ci\^encia e a Tecnologia, FCT, under the project UID/04674/2025.

\end{document}